 \documentclass[10pt]{amsart}
\usepackage{amsmath}
\usepackage{amsthm}
\usepackage{amssymb}
\usepackage{amsfonts}
\usepackage{latexsym}
\usepackage{verbatim}

\newcommand{\bmat}[1]{\left[\begin{array}{#1}}
\newcommand{\bbmat}[0]{\end{array}\right]}

\newcommand{\pmat}[0]{\begin{pmatrix}}
\newcommand{\ppmat}[0]{\end{pmatrix}}

\newcommand{\R}[0]{\mathbb{R}}

\newcommand{\prv}[1]{\mathcal{P}_{#1}}

\newcommand{\vp}{\varepsilon}

\newtheorem{thm}{Theorem}

\newtheorem{cor}[thm]{Corollary}

\newtheorem{lem}[thm]{Lemma}
\theoremstyle{definition}
\newtheorem{defin}[thm]{Definition}

\newtheorem{Example}[thm]{Example}

\newtheorem{Question}[thm]{Question}

\allowdisplaybreaks

    \title{Moving Parseval frames for vector bundles}
    \date{\today}
    \author{D. Freeman, D. Poore, A. R. Wei, M. Wyse}

\address{Department of Mathematics, University of Texas at Austin, 1 University Station
C1200, Austin TX 78712-0257} \email{freeman@math.utexas.edu}
\address{Department of Mathematics, Pomona College, Claremont, CA 92711}
\email{dep02007@mymail.pomona.edu}
\address{Department of Mathematics, Case Western Reserve University, Cleveland, OH 44106}
\email{arw27@case.edu}
\address{Department of Mathematics, Pomona College, Claremont, CA
91711} \email{mkw02007@mymail.pomona.edu}
\thanks{
The research of all four authors was supported by the "Wavelets and
Matrix Analysis" REU at Texas A\&M which was supported by the
National Science Foundation.  The first author was supported by a
grant from the National Science Foundation.}

\thanks{2010 \textit{Mathematics Subject Classification}: 42C15, 57R22}

\begin{document}

\begin{abstract}

Parseval frames can be thought of as redundant or linearly dependent
coordinate systems for Hilbert spaces, and have important
applications in such areas as signal processing, data compression,
and sampling theory. We extend the notion of a Parseval frame for a
fixed Hilbert space to that of a moving Parseval frame for a vector
bundle over a manifold. Many vector bundles do not have a moving
basis, but in contrast to this every vector bundle over a
paracompact manifold has a moving Parseval frame. We prove that a
sequence of sections of a vector bundle is a moving Parseval frame
if and only if the sections are the orthogonal projection of a
moving orthonormal basis for a larger vector bundle. In the case
that our vector bundle is the tangent bundle of a Riemannian
manifold, we prove that a sequence of vector fields is a Parseval
frame for the tangent bundle of a Riemannian manifold if and only if
the vector fields are the orthogonal projection of a moving
orthonormal basis for the tangent bundle of a larger Riemannian
manifold.


\end{abstract}\maketitle

\section{Introduction}\label{S:1}

Frames for Hilbert spaces are essentially redundant coordinate
systems.  That is, every vector can be represented as a series of
scaled frame vectors, but the series is not unique. Though this
redundancy is not necessary in a coordinate system, it can actually
be very useful. In particular, frames have played important roles in
modern signal processing after originally being applied in 1986 by
Daubechies, Grossmann, and Meyer \cite{DGM}. Besides being important
for their real world applications, frames are also interesting for
both their analytic and geometric properties
\cite{BF},\cite{BCPS}\cite{DFKLOW},\cite{HL} as well as their
connection to the famous Kadison-Singer problem
\cite{CCLV},\cite{W}.

A sequence of vectors $(x_i)$ in a Hilbert space $H$ is called a
{\em frame} for $H$ if there exists constants $A,B>0$ such that
$$A\|x\|^2\leq\sum |\langle x_i,x\rangle|^2\leq B\|x\|^2\qquad\textrm{ for
all } x\in H.$$ The constants $A,B$ are called the {\em frame
bounds}. The frame is called {\em tight} if $A=B$, and is called
{\em Parseval} if $A=B=1$.  The name Parseval was chosen because
$A=B=1$ if and only if the frame satisfies Parseval's identity. That
is, a sequence of vectors $(x_i)$ in a Hilbert space $H$ is a
Parseval frame for $H$ if and only if
$$\sum \langle x_i,x\rangle x_i=x \textrm{ for all }x\in H.$$  This useful reconstruction formula
 follows from the dilation theorem of Han and Larson \cite{HL}.
They proved that if $(x_i)$ is a Parseval frame for a Hilbert space
$H$, then $(x_i)$ is the orthogonal projection of an orthonormal
basis for a larger Hilbert space which contains $H$ as a subspace.
It is easy to see that the orthogonal projection of an orthonormal
basis is a Parseval frame, and thus the dilation theorem
characterizes Parseval frames as orthogonal projections of
orthonormal bases.

In differential topology and differential geometry, the word frame
has a different meaning. A moving frame for the tangent bundle of a
smooth manifold is essentially a basis for the tangent space at each
point in the manifold which varies smoothly over the manifold.  In
other words, a moving frame for the tangent bundle of an
$n$-dimensional smooth manifold is a set of $n$ linearly independent
vector fields. These two different definitions for the word "frame",
naturally lead one to question how they are related.  We will
combine the concepts by studying Parseval frames which vary smoothly
over a manifold, which we formally define below.
\begin{defin}\label{frame}
Let $\pi:E\rightarrow M$ be a rank n-vector bundle over a smooth
manifold $M$ with a given inner product $\langle\cdot,\cdot\rangle$.
Let $k\geq n$, and $f_i:M\rightarrow E$ be a smooth section of $\pi$
for all $1\leq i\leq k$. We say that $(f_i)_{i=1}^k$ is a {\em
moving Parseval frame} for $\pi$ if $(f_i(x))_{i=1}^k$ is a Parseval
frame for the fiber $\pi^{-1}(x)$ for all $x\in M$.  That is, for
all $x\in M$,
$$y=\sum_{i=1}^k \langle y,f_i(x)\rangle f_i(x)\qquad\textrm{ for all
}y\in\pi^{-1}(x).
$$
\end{defin}
A Parseval frame for a fixed Hilbert space can be constructed by
projecting an orthonormal basis, and thus the natural way to
construct moving Parseval frames is to project moving orthonormal
bases.  For instance, the two-dimensional sphere $S^2$ does not have
a nowhere-zero vector field, and hence cannot have a moving
orthonormal basis for its tangent space.  However, if we consider
$S^2$ as the unit sphere in $\R^3$ and $(e_i)_{i=1}^3$ as the
standard unit vector basis for $\R^3$, then at each point $p\in S^2$
we may project $(e_i)_{i=1}^3$ onto the tangent space $T_p(S^2)$,
giving us a moving Parseval frame of three vectors for $T S^2$.  As
every vector bundle over a para-compact manifold is a subbundle of a
trivial bundle, we may project the basis for the trivial bundle onto
the sub-bundle and obtain that every vector bundle over a
para-compact manifold has a moving Parseval frame. Thus in contrast
to moving bases, we have that moving Parseval frames always exist.
The natural general questions to consider are then: When do moving
Parseval frames with particular structure exist? How do theorems
about Parseval frames generalize to the vector bundle setting?, and
How can we construct nice moving Parseval frames for vector bundles
in the absence of moving bases? Our main results are the following
theorems which extend the dilation theorem of Han and Larson to the
context of vector bundles.  The proofs will be given in Section
\ref{S:3}.

\begin{thm}\label{VB}
Let $\pi_1:E_1\rightarrow M$ be a rank $n$ vector bundle over a
paracompact manifold $M$ with a moving Parseval frame
$(f_i)_{i=1}^k$.  There exists a rank $k-n$ vector bundle
$\pi_2:E_2\rightarrow M$ with a moving Parseval frame
$(g_i)_{i=1}^k$ so that $(f_i\oplus g_i)_{i=1}^k$ is a moving
orthonormal basis for the vector bundle $\pi_1\oplus\pi_2:E_1\oplus
E_2\rightarrow M$.
\end{thm}

If $M$ is a Riemannian manifold with a moving Parseval frame for
$TM$, we may apply Theorem \ref{VB} to obtain a vector bundle
containing $TM$ with a moving orthonormal basis which projects to
the moving Parseval frame.  However, if we start with a moving
Parseval frame for a tangent bundle, we want to end up with a moving
orthonormal basis for a larger tangent bundle which projects to the
moving Parseval frame. This way we would remain in the class of
tangent bundles, instead of general vector bundles.  The following
theorem states that we can do this.

\begin{thm}\label{T:main}
Let $M^n$ be an $n$-dimensional Riemannian manifold and
$(f_i)_{i=1}^k$ be a moving Parseval frame for $TM$ for some $k\geq
n$.  There exists a $k$-dimensional Riemannian manifold $N^k$ with a
moving orthonormal basis $(e_i)_{i=1}^k$ for $TN$ such that $N^k$
contains $M^n$ as a submanifold and $P_{T_x M}e_i(x)=f_i(x)$ for all
$x\in M^n$ and $1\leq i\leq k$, where $P_{T_x M}$ is orthogonal
projection from $T_x N$ onto $T_x M$.
\end{thm}

Though the concept of a moving Parseval frame seems natural to
consider, we are aware of only one paper on the subject. In 2009, P.
Kuchment proved in his Institute of Physics select paper that
particular vector bundles over the torus, which arise in
mathematical physics, have natural moving Parseval frames but do not
have moving bases \cite{K}.  The relationship between frames for
Hilbert spaces and manifolds was also considered in a different
context by Dykema and Strawn, who studied the manifold structure of
collections of Parseval frames under certain equivalent classes
\cite{DySt}.

We will use the term {\em inner product} on a vector bundle
$\pi:E\rightarrow M$ to mean a positive definite symmetric bilinear
form.  All of our theorems will concern vector bundles with a given
inner product. In the case that our vector bundle is the tangent
bundle of a Riemannian manifold, we will take the inner product to
be the Riemannian metric. For terminology and background on vector
bundles and smooth manifolds see \cite{L}, for terminology and
background on frames for Hilbert spaces see \cite{C} and
\cite{HKLW}.

The majority of the research contained in this paper was conducted
at the 2009 Research Experience for Undergraduates in Matrix
Analysis and Wavelets organized by Dr David Larson.  The first
author was a research mentor for the program, and the second, third,
and fourth authors were participants.  We sincerely thank Dr Larson
for his advice and encouragement.

\section{Preliminaries and Examples}\label{S:2}

Our goal is to study moving Parseval frames and extend theorems
about fixed Parseval frames for Hilbert spaces to moving Parseval
frames for vector bundles.  To do this, we will first need to define
some notation and recall some useful characterizations of Parseval
frames for Hilbert spaces in terms of matrices. For $F=(f_i)_{i=1}^k
\in \oplus_{i=1}^k\R^n$ and $(u_i)_{i=1}^n$ a fixed orthonormal
basis for $\R^n$, we denote $[F]_{n\times k}$ to be the matrix whose
column vectors with respect to the basis $(u_i)_{i=1}^n$ are given
by $(f_i)_{i=1}^k$. For $F=(f_i)_{i=1}^k \in \oplus_{i=1}^k\R^n$,
$G=(f_i)_{i=1}^k \in \oplus_{i=1}^k\R^m$, we define $F \oplus
G=(f_i\oplus g_i)_{i=1}^k \in \oplus_{i=1}^k\R^{n+m}$. If $k>n$,
$(u_i)_{i=1}^n$ is a fixed orthonormal basis for $\R^n$ and
$(u_i)_{i=n+1}^k$ is a fixed orthonormal basis for $\R^{k-n}$, then
the matrix $[F \oplus G]_{(n+m)\times k}$ given with respect to
$(u_i)_{i=1}^k$ will be formed by appending the column vectors
$(g_i)_{i=1}^k$ to the column vectors $(f_i)_{i=1}^k$. In other
words,
$[F \oplus G]_{(n+m)\times k} = \pmat f_1 & \cdots & f_k \\
g_1 & \cdots & g_k \ppmat$.  This matrix framework allows us to
provide a simple proof of the Han-Larson dilation theorem for
Parseval frames for $\R^n$.  In a later section we will extend this
proof to vector bundles.

\begin{thm}\cite{HL}\label{HL dilation}
If $k>n$, and $(f_i)_{i=1}^k$ is a Parseval frame for $\R^n$, then
there exists a Parseval frame $(g_i)_{i=1}^k$ for $\R^{k-n}$ such
that $(f_i\oplus g_i)_{i=1}^k$ is an orthonormal basis
 for $\R^n\oplus\R^{k-n}$.
\end{thm}

\begin{proof}
We denote the unit vector basis for $\R^n$ by $(u_i)_{i=1}^n$.  Let
$F=(f_i)_{i=1}^k$ and let $[F]_{n\times k}$ be the matrix whose
column vectors with respect to $(u_i)_{i=1}^n$ are given by
$(f_i)_{i=1}^k$.  If $1\leq p,q\leq n$, then the inner product of
the $p$th row of $T$ with the $q$th row of $T$ is given by
$\sum_{i=1}^k\langle f_i,u_p \rangle\langle f_i,u_q \rangle$.  We
now use the following equality.
\begin{align*} 2=\sum_{i=1}^k\langle f_i,u_p+u_q\rangle^2=&\sum_{i=1}^k\langle
f_i,u_p\rangle^2+\sum_{i=1}^k\langle
f_i,u_q\rangle^2+2\sum_{i=1}^k\langle f_i,u_p\rangle\langle
f_i,u_q\rangle\\
=&2+2\sum_{i=1}^k\langle f_i,u_p\rangle\langle f_i,u_q\rangle
\end{align*}
Thus we have that $\sum_{i=1}^k\langle f_i,u_p \rangle\langle
f_i,u_q \rangle=0$, and hence the rows of $[F]_{n\times k}$ are
orthonormal. We can thus choose $G=(g_i)_{i=1}^k\subset \R^{k-n}$
such that the rows of $[F\oplus G]_{k\times k}$ are orthonormal.
Thus the column vectors $(f_i\oplus g_i)_{i=1}^k$ of $[F\oplus
G]_{k\times k}$ form an orthonormal basis for $\R^n\oplus\R^{k-n}$.
We have that $(g_i)_{i=1}^k$ must be a Parseval frame for $\R^{k-n}$
as it is the orthogonal projection of the orthonormal basis
$(f_i\oplus g_i)_{i=1}^k$.
\end{proof}

As shown in the proof of Theorem \ref{HL dilation}, a sequence of
vectors $F=(f_i)_{i=1}^k\subset \R^n$ is a Parseval frame for $\R^n$
if and only if the matrix $[F]_{n\times k}$ has orthonormal rows.
The dilation theorem gives that Parseval frames are exactly
orthogonal projections of orthonormal bases.  It is then immediate
that the orthogonal projection of a moving orthonormal basis is a
moving Parseval frame.

\begin{thm}\label{projection thm}
Let $k\geq n$ and let $\pi: E\rightarrow M$ be a rank $k$ vector
bundle with an inner product $\langle\cdot,\cdot\rangle$ and moving
orthonormal basis $(e_i)_{i=1}^k$.  If $\pi|_{E_0}:E_0\rightarrow M$
is a rank $n$ sub-bundle, then $(P_{E_0}e_i)_{i=1}^k$ is a moving
Parseval frame for $\pi|_{E_0}:E_0\rightarrow M$, where
$P_{E_0}(e_i(x))$ is the orthogonal projection of $e_i(x)$ onto the
fiber $\pi|_{E_0}^{-1}(x)$ for all $x\in M$.
\end{thm}
\begin{proof}
As $\pi|_{E_0}:E_0\rightarrow M$ is a subbundle of $\pi:
E\rightarrow M$, we have that $P_{E_0}:E\rightarrow E_0$ is
continuous.  Furthermore, for all $1\leq i\leq k$, we have that
$\pi|_{E_0}(P_{E_0}(e_i(x)))=x$ for all $x\in M$.  Thus $P_{E_0}e_i$
is a section of $\pi|_{E_0}:E_0\rightarrow M$ for all $1\leq i\leq
k$. $(P_{E_0}e_i(x))_{i=1}^k$ is a Parseval frame for
$\pi|_{E_0}^{-1}(x)$ for all $x\in M$, as it is the orthogonal
projection of an orthonormal basis.  Thus $(P_{E_0}e_i(x))_{i=1}^k$
is a moving Parseval frame for $\pi|_{E_0}:E_0\rightarrow M$.
\end{proof}

By applying Theorem \ref{projection thm} to the tangent bundle of a
smooth manifold, we obtain the following corollary for Riemannian
manifolds.

\begin{cor}\label{projection cor}
Let $k\geq n$ and let $N$ be a $k$-dimensional Riemannian manifold
with a moving orthonormal basis $(e_i)_{i=1}^k$ for its tangent
bundle $TN$. If $M\subset N$ is a smooth sub-manifold, then
$(P_{TM}e_i)_{i=1}^k$ is a moving Parseval frame for $TM$, where
$P_{TM}(e_i(x))$ is the orthogonal projection of $e_i(x)\in T_x N$
onto $T_{x}M$ for all $x\in M$ and $1\leq i\leq k$.
\end{cor}
\begin{proof}
Let $\pi:TN\rightarrow N$ be the tangent bundle for $N$.  Then
$(e_i|_M)_{i=1}^k$ is a moving orthonormal basis for the vector
bundle $\pi|_{\pi^{-1}(M)}:\pi^{-1}(M)\rightarrow M$, which contains
$TM$ as a sub-bundle.   We may thus apply Theorem \ref{projection
thm}.
\end{proof}

For example, the two dimensional sphere $S^2$ does not have a moving
basis for its tangent space, as it does not have a nowhere zero
vector field.  However, if we consider $\R^3$ to be a Riemanian
manifold with the Riemannian metric given by the dot product, then
 $(e_i)_{i=1}^3$ is a moving orthonormal basis for $T\R^3$, where $e_1(x,y,z)=(1,0,0)$,
$e_2(x,y,z)=(0,1,0)$, and $e_3(x,y,z)=(0,0,1)$ for all $(x,y,z)\in
\R^3$.  We can then project $(e_i)_{i=1}^3$ onto the tangent bundle
of the unit sphere to obtain a moving Parseval frame $(f_i)_{i=1}^3$
for $TS^2$.  In this case, $(f_i)_{i=1}^3$ will be defined by
$f_1(x,y,z)=(1-x^2,-xy,-xz)$, $ f_2(x,y,z)=(-xy,1-y^2,-yz)$, and
$f_3(x,y,z)=(-xz,-yz,1-z^2)$ for all $(x,y,z)\in S^2$.

If $M$ is an $n$ dimensional smooth manifold, and $\phi:M\rightarrow
N$ is an embedds into a $k$ dimensional Riemannian manifold $N$ with
a moving orthonormal basis for $TN$, then we can project the moving
orthonormal basis onto $T\phi(M)$ and then pull it back to obtain a
moving Parseval frame for $TM$ of $k$ vectors. Furthermore, this may
be done if $\phi$ is only an immersion instead of an embedding. The
Whitney immersion theorem gives that for all $n\geq 2$, every $n$
dimensional paracompact smooth manifold immerses in $\R^{2n-1}$.
Thus every $n$ dimensional smooth manifold has a moving Parseval
frame for its tangent bundle of $2n-1$ vectors.  When considering
$n=2$, we have that $\R^{2}$, the cylinder and the torus are the
only two dimensional manifolds with continuous moving basis for its
tangent bundle. However, every two dimensional paracompact smooth
manifold has a moving Parseval frame of three vectors obtained by
immersing the manifold in $\R^3$. Unfortunately, obtaining a moving
Parseval frame in this way often does not lend us much intuition
about the space in question. We present here an intuitive moving
Parseval frame for the tangent bundles of the M$\ddot{o}$bius strip
and Klein bottle which cannot be obtained by immersing in $\R^3$
with the usual orthonormal basis, but which reflects the topology of
the surface.

\begin{Example}

We represent the M$\ddot{o}$bius strip and Klein bottle in the
standard way with the square $[0,1]\times[0,1]$, where  we identify
the top and bottom according to $(x,1)\equiv (1-x,0)$ for all $0\leq
x\leq 1$, and for the Klein bottle we identify the sides according
to $(1,y)\equiv (0,y)$ for all $0\leq y\leq1$, as seen in Figure 1.

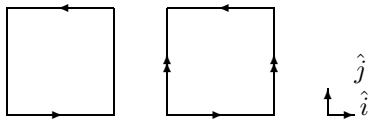
\begin{figure}[hbt]
\setlength{\unitlength}{0.14in} 
\centering 
\begin{picture}(14,4.5) 

\put(2.1,0){\vector(1,0){0}} \put(0,0){\line(1,0){4}}
\put(1.9,4){\vector(-1,0){0}} \put(0,4){\line(1,0){4}}
\put(0,0){\line(0,1){4}} \put(4,0){\line(0,1){4}}

\put(8.1,0){\vector(1,0){0}} \put(6,0){\line(1,0){4}}
\put(7.9,4){\vector(-1,0){0}} \put(6,4){\line(1,0){4}}
\put(6,2.3){\vector(0,1){0}} \put(6,2){\vector(0,1){0}}
\put(6,0){\line(0,1){4}} \put(10,2){\vector(0,1){0}}
\put(10,2.3){\vector(0,1){0}} \put(10,0){\line(0,1){4}}

\put(12,0){\vector(1,0){1}} \put(13.2,0) {$\hat{i}$}
\put(12,0){\vector(0,1){1}} \put(13,1.4) {$\hat{j}$}

\end{picture}
\caption{Aligned M$\ddot{o}$bius Strip (left) and Klein bottle
(right)}
\end{figure}

For all $(x, y)\in[0,1]\times[0,1]$, let
\begin{equation*}
f_1(x,y) = (cos(\pi y),0)\quad f_2(x,y) = (sin(\pi y),0)\quad
f_3(x,y) = (0,1).
\end{equation*}
It is easy to see that $(f_i)_{i=1}^3$ is a moving Parseval frame
for both the M$\ddot{o}$bius strip and the Klein bottle, which
naturally shows the twist in their topology.

\end{Example}

Given a vector bundle $\pi_1:E_1\rightarrow M$, it is a classic
problem in differential topology to find a vector bundle
$\pi_2:E_2\rightarrow M$ so that $\pi_1\oplus\pi_2:E_1\oplus
E_2\rightarrow M$ has a moving basis. This is of course closely
related to our work. Before proving Theorem \ref{VB}, we need to
show that our condition that $\pi_1\oplus\pi_2:E_1\oplus
E_2\rightarrow M$ has an orthonormal basis which projects to a given
Parseval frame for $\pi_1:E_1\rightarrow M$ is in fact stronger in
general than the condition that $\pi_1\oplus\pi_2:E_1\oplus
E_2\rightarrow M$ simply has a basis. Thus the dilation theorems for
moving Parseval frames do not follow as corollaries from known
results in differential topology.  This will be illustrated by the
following simple example.

\begin{Example}\label{Sphere example}
We define a moving Parseval frame $(f_i)_{i=1}^3$ for the vector
bundle $S^2\times\R$ by $f_1\equiv1$ and $f_2\equiv f_3\equiv0$. The
normal bundle to $TS^2\subset T\R^3$ is simply $S^2\times\R$, and
thus $(S^2\times\R)\oplus TS^2\cong S^2\times\R^3$ has a moving
basis. However, we claim that there does not exist a moving
 basis $(e_i)_{i=1}^3$ for $(S^2\times\R)\oplus TS^2$
such that $P_{S^2\times\R}e_i=f_i$ for all $i=1,2,3$.  Indeed, if
$P_{S^2\times\R}e_2=f_2=0$ then $P_{TS^2}e_2$ is nowhere zero.
However, $S^2$ does not have a nowhere zero vector field, and thus
we have a contradiction.
\end{Example}

We have a case of two vector bundles $\pi_1:E_1\rightarrow M$ and
$\pi_2:E_2\rightarrow M$ and  a moving Parseval frame
$(f_{i})_{i=1}^k$ for $\pi_1$ which does not dilate to a moving
basis for $\pi_1\oplus\pi_2:E_1\oplus E_2\rightarrow M$, even though
$\pi_1\oplus\pi_2:E_1\oplus E_2\rightarrow M$ has a moving basis of
$k$ vectors. This motivates the following question.  What properties
of a moving Parseval frame $(f_i)_{i=1}^k$ for a vector bundle
$\pi_1:E_1\rightarrow M$ would guarantee that if
$\pi_1\oplus\pi_2:E_1\oplus E_2\rightarrow M$ has a moving basis of
$k$ vectors, then $\pi_1\oplus\pi_2:E_1\oplus E_2\rightarrow M$ has
a moving orthonormal basis which projects to $(f_i)_{i=1}^k$? The
following theorem answers this question when $k=n+1$, where $n$ is
the rank of the vector bundle $\pi_1$.

\begin{thm}\label{T:n+1}
Let $(f_i)_{i=1}^{n+1}$ be a moving Parseval frame for a rank $n$
vector bundle $\pi_1:E_1\rightarrow M$.   If $\pi_2:E_2\rightarrow
M$ is a rank 1 vector bundle such that $\pi_1\oplus\pi_2:E_1\oplus
E_2\rightarrow M$ has a moving basis, then $\pi_2: E_2\rightarrow M$
has a moving Parseval frame $(g_i)_{i=1}^{n+1}$ such that
$(f_i\oplus g_i)_{i=1}^{n+1}$ is a moving orthonormal basis for
$\pi_1\oplus\pi_2:E_1\oplus E_2\rightarrow M$.
\end{thm}
\begin{proof}
 If $\pi_1\oplus\pi_2:E_1\oplus
E_2\rightarrow M$ has a moving orthonormal basis
$(e_i)_{i=1}^{n+1}$, then the determinant of an operator or matrix
with respect to $(e_i)_{i=1}^{n+1}$ varies smoothly over
$\pi_1\oplus\pi_2:E_1\oplus E_2\rightarrow M$.  If $T$ is an
operator or matrix, we will denote $\det_e(T)$ to be the determinant
of $T$ with respect to $(e_i)_{i=1}^{n+1}$.

We will first prove the result locally, and then we will show that
our local choice can actually be made globally.  By Lemma
\ref{LemmaLocal}, for each $x\in M$ there exists $\vp_x>0$ and a
smoothly varying frame $(g_{x,i})_{i=1}^{n+1}$ for
$\pi_2|_{\pi_2^{-1}(B_{\vp_x}(x))}$ such that $(f_i\oplus
g_{x,i})_{i=1}^{n+1}$ is a moving orthonormal basis for
$\pi_1|_{\pi_1^{-1}(B_{\vp_x}(x))}\oplus\pi_2|_{\pi_2^{-1}(B_{\vp_x}(x))}$.
For each $x,y\in M$, we denote $[f_i(y)\oplus g_{x,i}(y)]_{k\times
k}$ to be the matrix with respect to the basis
$(e_i(x))_{i=1}^{n+1}$ whose column vectors are $(f_i(y)\oplus
g_{x,i}(y))_{i=1}^{n+1}$.  It is easy to see that $(f_i(y)\oplus
g_{x,i}(y))_{i=1}^{n+1}$ is an orthonormal basis if and only if
$(f_i(y)\oplus -g_{x,i}(y))_{i=1}^{n+1}$ is an orthonormal basis.
Thus without loss of generality, we may assume that
$(g_{x,i})_{i=1}^{n+1}$ has been chosen such that
$\det_e[f_i(x)\oplus g_{x,i}(x)]_{k\times k}=1$ for all $x\in X$,
and hence $\det_e[f_i(y)\oplus g_{x,i}(y)]_{k\times k}=1$ for all
$x\in X$ and $y\in B_{\vp_x}(x)$ as $\det_e$ is continuous.
 As the span of
$(f_i(y)\oplus 0)_{i=1}^{n+1}$ has co-dimension 1, there is exactly
one choice for $(g_{x,i}(y))_{i=1}^{n+1}$ such that $(f_i(y)\oplus
g_{x,i}(y))_{i=1}^{n+1}$ is orthonormal and $\det_e[f_i(y)\oplus
g_{x,i}(y)]_{k\times k}=1$.  Thus our locally smooth choice was
unique, and hence is smooth globally.
\end{proof}

\section{Proofs of Dilation Theorems}\label{S:3}


We denote the set of all Parseval frames of $k$ vectors for $\R^n$
by $\prv{k,n}$. Specifically,
$\prv{k,n}=\{(f_i)_{i=1}^k\in\oplus_{i=1}^k\R^n:
(f_i)_{i=1}^k\textrm{ is a Parseval frame for }\R^n\}$.  In order to
study moving Parseval frames over smooth manifolds, we need to first
establish that $\prv{k,n}$ itself is a smooth manifold.

\begin{thm}
 For every $k\geq n$, the set
$\prv{k,n}$ is a smooth submanifold of $\oplus_{i=1}^k\R^n$ of
dimension $kn-n(n+1)/2$.
\end{thm}
\begin{proof}
If ${\bf F}=(f_i)_{i=1}^k\in \oplus_{i=1}^k\R^n$ then the positive
self-adjoint operator defined by $S_{\bf F}(x)=\sum_{i=1}^k\langle
x,f_i\rangle f_i$ for all $x\in\R^n$ is called the frame operator
for $(f_i)_{i=1}^k$.  It is clear that the map $\phi:
\oplus_{i=1}^k\R^n\rightarrow B(\R^n)$ given by $\phi({\bf
F})=S_{\bf F}$ is smooth and $\prv{k,n}=\phi^{-1}(Id)$. The set of
self-adjoint operators on $\R^n$ is naturally diffeomorphic to
$\R^{n(n+1)/2}$ as seen by fixing a basis and representing the
self-adjoint operators by symmetric matrices.  The positive definite
self-adjoint operators are an open subset of the self-adjoint
operators and thus form a smooth manifold.  By Sard's theorem there
exists a positive definite self-adjoint operator $A$ which is a
regular value of $\phi$, and thus $\phi^{-1}(A)$ is a smooth
submanifold of $\oplus_{i=1}^k\R^n$.  We have that $\phi^{-1}(A)$
has dimension $kn-n(n+1)/2$ as the manifold of positive definite
self adjoint operators has dimension $n(n+1)/2$ and $\R^{kn}$ has
dimension $kn$. We define a diffeomorphism $\psi_A:
\oplus_{i=1}^k\R^n\rightarrow \oplus_{i=1}^k\R^n$ by
$\psi_A((f_i)_{i=1}^k)=(A^{-1/2}f_i)_{i=1}^k$. As $A$ is self
adjoint we have that
$$\phi\circ\psi_A ({\bf F})(x)=\sum_{i=1}^k\langle x,A^{-\frac{1}{2}}f_i\rangle
A^{-\frac{1}{2}}f_i=A^{-\frac{1}{2}}\sum_{i=1}^k\langle
A^{-\frac{1}{2}}x,f_i\rangle f_i=A^{-\frac{1}{2}}\phi({\bf
F})A^{-\frac{1}{2}}x.
$$
Thus $\phi\circ\psi_A ({\bf F})=A^{-\frac{1}{2}}\phi({\bf
F})A^{-\frac{1}{2}}$, and hence $\psi_A(\phi^{-1}A)=\phi^{-1}(Id)$.
We conclude that $\prv{k,n}=\phi^{-1}(Id)$ is diffeomorphic to
$\phi^{-1}(A)$ and is hence a smooth submanifold of
$\oplus_{i=1}^k\R^n$ of dimension $kn-n(n+1)/2$.
\end{proof}

We note that the same proof gives that the set of frames of
$k$-vectors in $\R^n$ with a given invertible frame operator is a
smooth sub-manifold of $\oplus_{i=1}^k\R^n$.  We now prove that
locally, we can smoothly choose complementary frames for Parseval
frames.  Note that for $k\geq1$, $\prv{k,k}$ is the collection of
all orthonormal bases for $\R^k$.

\begin{lem}\label{LemmaLocal}

For every $k\geq n$ and $F \in \prv{k,n}$ there exists some
$\varepsilon
> 0$ and a smooth map
$\phi:B_\vp(F)\cap\prv{k,n}\rightarrow\prv{k,k-n}$ such that $G
\oplus \phi(G) \in \prv{k,k}$ for all $G\in B_\vp(F)\cap\prv{k,n}$.
\end{lem}
\begin{proof}
We choose $H\in\prv{k,k-n}$ such that $F\oplus H\in\prv{k,k}$ and
fix an orthonormal basis for $\R^n$. As mentioned earlier, this is
equivalent to the matrix $[F\oplus H]_{k\times k}$ being unitary.
The set of invertible matrices is open, and thus there exists
$\vp>0$ such that $[G\oplus H]_{k\times k}$ is invertible for all
$G\in B_\vp(F)\cap\prv{k,n}$. For each $G\in B_\vp(F)$, we apply the
Gram-Schmidt procedure to the rows of $[G\oplus H]_{k\times k}$,
where the procedure is applied to the rows of $[G]_{n\times k}$
before the rows of $[H]_{(k-n)\times k}$. As $G\in\prv{k,n}$, the
rows of $[G]_{n\times k}$ are orthonormal. Thus the Gram-Schmidt
procedure when applied to $[G\oplus H]_{k\times k}$ will leave the
rows contained in $[G]_{n\times k}$ fixed, and hence  the matrix
resulting from applying the Gram-Schmidt procedure will be of the
form $[G\oplus\phi(G)]_{k\times k}$ for some
$\phi(G)\in\prv{k,k-n}$. Furthermore, the map
$\phi:\prv{k,n}\rightarrow\prv{k,k-n}$ is smooth as the Gram-Schmidt
procedure is smooth when applied to the rows of any set of
invertible matrices.

\end{proof}

The frame given in Example \ref{Sphere example} shows that Lemma
\ref{LemmaLocal} is not true globally for $k=3$ and $n=1$.  However,
the proof of Theorem \ref{T:n+1} gives that Lemma \ref{LemmaLocal}
is true globally for $k=n+1$.  We are now ready to prove Theorem
\ref{VB}.
\begin{proof}[Proof of Theorem \ref{VB}]

Let $\pi_1:E_1\rightarrow M$ be a rank $n$ vector bundle over a
paracompact manifold $M$ with a moving Parseval frame
$(f_i)_{i=1}^k$. By Lemma \ref{LemmaLocal}, we can locally choose a
complementary moving Parseval frame.  That is, for each $x\in M$,
there exists $\vp_x>0$ and a moving Parseval frame
$(g_{x,i})_{i=1}^k$ for the trivial vector bundle
$\pi_{x}:B_{\vp_x}(x)\times\R^{k-n}\rightarrow B_{\vp_x}(x)$ such
that $(f_i\oplus g_i)_{i=1}^k$ is a moving orthonormal basis for the
vector bundle $\pi_1|_{\pi_1^{-1}(B_{\vp_x}(x))}\oplus\pi_{x}$. The
collection of sets $\{B_{\vp_x}(x)\}_{x\in M}$ is an open cover of
$M$, and thus there is a partition of unity $\{\psi_a\}_{a\in A}$
subordinate to a locally finite open refinement $\{U_a\}_{a\in A}$.
We thus have for each $a\in A$ a moving Parseval frame
$(g_{a,i})_{i=1}^k$ for the trivial vector bundle
$\pi_a:U_a\times\R^{k-n}\rightarrow U_a$ such that $(f_i\oplus
g_{a,i})_{i=1}^k$ is a moving orthonormal basis for
$\pi_1|_{\pi_1^{-1}(U_a)}\oplus \pi_a$.
 We use the partition of unity to extend $(g_{a,i})_{i=1}^k$ to
all of $M$ by defining $g_i=\bigoplus_{a\in A}\psi_a^{1/2} g_{a,i}$
for all $1\leq i\leq n$, where we set $g_{a,i}(x)=0$ if $x\not\in
U_a$.  Thus $g_i$ is a smooth section of the trivial vector bundle
$\pi_t:M\times\oplus_{a\in A}\R^{k-n}\rightarrow M$ for all $1\leq
i\leq k$. The following simple calculations show that $(f_i(x)\oplus
g_i(x))_{i=1}^k$ is an orthonormal set of vectors in
$\pi_1^{-1}(x)\oplus\bigoplus_{a\in A} \R^{k-n}$ for all $x\in M$.
For all $1\leq i,j\leq k$, we have the following calculation.
\begin{align*}
\langle f_i(x)\oplus g_i(x),f_j(x)\oplus g_j(x)\rangle&=\langle
f_i(x),f_j(x)\rangle+\sum_{a\in A}\psi_a \langle g_{a,i}(x),
g_{a,j}(x)\rangle\\
 &=\sum_{a\in A}\psi_a \big(\langle
f_i(x),f_j(x)\rangle+\langle g_{a,i}(x), g_{a,j}(x)\rangle\big)\\
&=\sum_{a\in A}\psi_a \langle f_i(x)\oplus g_{a,i}(x),f_j(x)\oplus
g_{a,j}(x)\rangle\\
&=\sum_{a\in A}\psi_a \delta_{i,j}=\delta_{i,j}
\end{align*}
Thus, $(f_i\oplus g_i)_{i=1}^k$ is a sequence of smooth orthonormal
sections of
 $\pi_1\oplus \pi_t$, and
hence $E:=span_{1\leq i\leq k,x\in M} f_i(x)\oplus g_i(x)$ is a
smooth manifold and we have an induced vector bundle
$\pi_E:E\rightarrow M$.
 We now show that $(f_j)_{1\leq j\leq k}$ being a moving Parseval
frame implies that $span_{1\leq j\leq k} f_j(x)\oplus0\subset
span_{1\leq j\leq k} f_j(x)\oplus g_j(x)$ for all $x\in M$. Indeed,
if $x\in M$ and $y\in\pi_1^{-1}(x)$ then we calculate the following.
$$
\|P_{\pi_E^{-1}(x)}y\oplus0\|^2=\sum_{i=1}^k<y\oplus0,f_i\oplus
g_i>^2 =\sum_{i=1}^k<y,f_i>^2=\|y\|^2
$$
Which implies that $y\oplus0=P_{\pi_E^{-1}(x)}y\oplus0$.  Thus we
conclude that $span_{1\leq j\leq k} f_j(x)\oplus0\subset span_{1\leq
j\leq k} f_j(x)\oplus g_j(x)$, and hence $p_1:E_1\rightarrow M$ is a
sub-bundle of $\pi_E:E\rightarrow M$. It is then possible to define
$\pi_2:E_2\rightarrow M$ as the orthogonal bundle of
$\pi_1:E_1\rightarrow M$ in $\pi_E:E\rightarrow M$. We have that
$f_i=P_{E_1}f_i\oplus g_i$ and thus by definition $g_i=P_{E_2}
f_i\oplus g_i$.  As the orthogonal projection of a moving
orthonormal basis onto a sub-bundle, $(g_i)_{i=1}^\infty$ is a
moving Parseval frame for $E_2$ and is a complementary frame for
$(f_i)_{i=1}^\infty$.
\end{proof}

The above proof takes an approach using local coordinates, and then
combines the pieces using a partition of unity.  As this is a common
technique in differential topology, the above construction is
valuable in that it could potentially be combined with other proofs
and constructions.  We present as well a second proof which avoids
local coordinates and is based on the original proof of the dilation
theorem of Han and Larson.

\begin{proof}[Second proof of Theorem \ref{VB}]
Let $\pi_1:E_1\rightarrow M$ be a rank $n$ vector bundle over a
paracompact manifold $M$ with a moving Parseval frame
$(f_i)_{i=1}^k$.  Let $\pi:M\times\R^k\rightarrow M$ denote the
trivial rank $k$ vector bundle over $M$, and let $(e_i)_{i=1}^k$ be
the moving unit vector basis for $\pi:M\times\R^k\rightarrow M$. We
define a bundle map $\theta: E_1\rightarrow M\times\R^k$ over $M$ by
$\theta(y)=\sum_{i=1}^k \langle y , f_i(\pi_1(y))\rangle
e_i(\pi_1(y))$.  As $(f_i(\pi_1(y))_{i=1}^k$ is a Parseval frame for
the fiber containing $y$, we have that $\|y\|^2=\sum_{i=1}^k \langle
y , f_i(\pi_1(y))\rangle^2=\|\theta(y)\|^2$.  Hence,
$\theta|_{\pi_1^{-1}(x)}$ is a linear isometric embedding of the
fiber for $\pi_1^{-1}(x)$ into the fiber $\pi^{-1}(x)$ all $x\in M$.
Thus, for convenience, we may identify the bundle
$\pi_1:E_1\rightarrow M$ with
$\pi|_{\theta(E_1)}:\theta(E_1)\rightarrow M$. In particular, we
have that
\begin{equation}\label{eq1}
\langle y,e_i(\pi(y))\rangle=\langle y,f_i(\pi(y))\rangle\textrm{
for all }y\in E_1.
\end{equation}
 Let
$P_1:M\times\R^k\rightarrow E_1$ be orthogonal projection, that is,
$P_1$ is the bundle map such that $P_1(y)$ is the orthogonal
projection of $y$ onto the fiber $\pi_1^{-1}(\pi(y))$. We now show
that $P_1\circ e_i=f_i$ for all $1\leq i\leq k$.  We let $x\in M$,
$y\in\pi_1^{-1}(x)$ and $1\leq i\leq k$, and obtain
\begin{align*}
\langle y, P_1(e_i(x))\rangle&= \langle y,
e_i(x)\rangle\qquad\textrm{ as }y\in E_1\\
&=\langle y, f_i(x)\rangle\qquad\textrm{ by the definition of }\theta\textrm{ as }x=\pi_1(y)\\
\end{align*}
Thus $\langle y, P_1(e_i(x))\rangle=\langle y, f_i(x)\rangle$ for
all $y\in\pi_1^{-1}(x)$, and hence $P_1(e_i(x))= f_i(x)$ for all
$x\in M$.
\end{proof}


We now consider the case where $M$ is a smooth paracompact manifold
with a moving Parseval frame for its tangent bundle $TM$. We may
apply Theorem \ref{VB} to obtain a vector bundle containing $TM$
with a moving orthonormal basis which projects to the moving
Parseval frame. However, we want the moving orthonormal basis to be
for a larger tangent bundle and not just a general vector bundle. To
do this, we will show that actually the total space of the vector
bundle given by Theorem \ref{VB} will be a Riemannian manifold with
an orthonormal basis which projects to the given Parseval frame for
$TM$.

\begin{proof}[Proof of Theorem \ref{T:main}]
We denote $\pi_1:TM\rightarrow M$ to be the tangent bundle. By
Theorem \ref{VB} there exists a rank $(k-n)$ vector bundle
$\pi_2:N\rightarrow M$ with a moving Parseval frame $(g_i)_{i=1}^k$
so that $(f_i\oplus g_i)_{i=1}^k$ is a moving orthonormal basis for
the vector bundle $\pi_1\oplus\pi_2$.  The manifold $N$ has
dimension $k$, as $M$ has dimension $n$ and the vector bundle
$\pi_2:N\rightarrow M$ has rank $k-n$.  For $p\in N$ and
$\gamma:\R\rightarrow N$ such that $\gamma(0)=p$, we have the
differential $D_\gamma\in TN$ defined by
$$D_\gamma(f)=\frac{d}{dt}f(\gamma(t))|_{t=0},
$$
for each smooth real valued $f$ defined on an open neighborhood of
$p$.  We define a smooth map $\Theta: N\times N\rightarrow N$ by
$$\Theta(p,q)=\sum_{i=1}^k\langle g_i(\pi_2(p)),p\rangle
g_i(\pi_2(q))\qquad\textrm{ for all }p,q\in N.$$ The map $\Theta$
has been constructed so that if $p$ and $q$ are contained in the
same fiber of $\pi_2:N\rightarrow M$, i.e. $\pi_2(q)=\pi_2(p)$, then
$\Theta(p,q)=p$.  Note that if $q_0,q_1\in N$ such that
$\pi_2(q_0)=\pi_2(q_1)$, then $\Theta(p,q_0)=\Theta(p,q_1)$. As
$\pi_2(\Theta(p,q))=\pi_2(q)$, we thus have that
$\Theta(p,q)=\Theta(p,\Theta(p,q))$ for all $p,q\in N$.
 We use $\Theta$ to define a smooth map $\psi:TN\rightarrow TN$ by setting
 for each smooth $\gamma:\R\rightarrow
N$, $\psi(D_\gamma)=D_{\gamma}-D_{\Theta(\gamma(0),\gamma)}$.  In
other words, if $\gamma:\R\rightarrow N$ is smooth and $f$ is a
smooth real valued function defined on an open neighborhood of
$\gamma(0)$, then
$$\psi(D_\gamma)(f)=\frac{d}{dt}f(\gamma(t))|_{t=0}-\frac{d}{dt}f\left(\sum_{i=1}^k\left\langle
g_i(\pi_2(\gamma(0))),\gamma(0)\right\rangle
g_i(\pi_2(\gamma(t)))\right)|_{t=0}.
$$
Note that if $D_{\gamma_0}=D_{\gamma_1}$ then
$\psi(D_{\gamma_0})=\psi(D_{\gamma_1})$, and that $\psi(a
D_{\gamma_0}+ D_{\gamma_1})=a \psi(D_{\gamma_0})+\psi(D_{\gamma_1})$
for all $a\in \R$ and smooth $\gamma_0,\gamma_1:\R\rightarrow N$.
Furthermore, if $\pi_2(\gamma)$ is constant, then
$\psi(D_\gamma)=D_\gamma$.

For each $p\in N$, we define a linear operator $\Phi:T_pN\rightarrow
T_{\pi_2(p)}M$ by $\Phi(D_\gamma)=D_{\pi_2\circ\gamma}$. Then
$\Phi(\psi(\gamma))=0$ for all smooth $\gamma:\R\rightarrow N$, as
$$\Phi(\psi(\gamma))=D_{\pi_2\circ\gamma}-D_{\pi_2(\Theta(\gamma(0),\gamma))}=D_{\pi_2\circ\gamma}-D_{\pi_2\circ\gamma}=0.$$
  For each $q\in \pi^{-1}_2(\pi_2(p))$ we
define $D_{q}$ as the derivative at $p$ in the direction of $q$.
That is, $D_{q}(f)=\frac{d}{dt}f(p+tq)|_{t=0}$, for each smooth real
valued $f$ defined on an open neighborhood of $p$.  We have that
$\psi(D_q)=D_q$ for all $q\in\pi^{-1}_2(\pi_2(p))$ as $\pi_2(p+tq)$
is constant with respect to $t$. Thus
$\{D_q\}_{\pi_2(q)=\pi_2(p)}=\psi(T_pN)=\Phi^{-1}(0)$ as
$\{D_q\}_{\pi_2(q)=\pi_2(p)}\subseteq\psi(T_pN)\subseteq\Phi^{-1}(0)$
and both spaces $\{D_q\}_{\pi_2(q)=\pi_2(p)}$ and $\Phi^{-1}(0)$ are
$(k-n)$-dimensional.  For each smooth real valued $f$ defined on an
open neighborhood of $p$, we denote $q_\gamma$ to be the unique
vector in $\pi_2^{-1}(\pi_2(p))$ such that
$D_{q_\gamma}=\psi(D_\gamma)$.


We define a smooth bundle map $\phi: TN\rightarrow
E(\pi_1\oplus\pi_2)$ by for $\gamma:\R\rightarrow N$, we set
$\phi(D_\gamma)=D_{\pi_2\circ\gamma}\oplus q_\gamma$.  Note that for
each $p\in N$, $\phi|_{T_pN}:T_pN\rightarrow T_{\pi_2(p)}M$ is an
isomorphism as $D_{\pi_2\circ\gamma}=0$ if and only if
$D_\gamma=D_{q_\gamma}$. Thus $\phi$ induces a Riemannian metric
$\langle\cdot,\cdot\rangle_N$ on $TN$ by $\langle
f,g\rangle_N=\langle \phi(f),\phi(g)\rangle$.  We have that
$(f_i\oplus g_i)_{i=1}^k$ is a moving orthonormal basis for
$\pi_1\oplus \pi_2$, and hence $(\phi^{-1}(f_i\oplus g_i))_{i=1}^k$
is a moving orthonormal basis for $TN$. If $p\in M$ and
$\gamma:\R\rightarrow M\subset N$ such that $\gamma(0)=p$ then
$\phi(D_\gamma)=D_\gamma\oplus0$.  Thus $\phi^{-1}(f_i(p)\oplus
0)=f_i(p)$ for all $1\leq i\leq k$ and $p\in M$, and hence
$f_i=P_{T_pM}\phi^{-1}(f_i\oplus g_i)$ for all $1\leq i\leq k$.
\end{proof}

We may apply Theorem \ref{T:main} to obtain the following corollary,
where we call a smooth manifold parallelizable if it has a moving
basis for its tangent bundle.

\begin{cor}
Let $M$ be a paracompact smooth manifold.  If $M$ immerses in a $k$
dimensional parallelizable paracompact smooth manifold, then $M$
embeds in a $k$ dimensional parallelizable paracompact smooth
manifold.
\end{cor}

\begin{proof}
Assume that $M$ immerses in a parallelizable smooth manifold $N$. We
may assign a Riemannian metric to $N$ such that the moving basis for
$TN$ is an orthonormal basis.  Projecting this orthonormal basis
then pulling back to $TM$ gives a moving Parseval frame for $TM$ of
$k$ vectors. Thus $M$ embeds in a $k$ dimensional parallelizable
smooth manifold by Theorem \ref{T:main}.
\end{proof}





\section{Open problems}\label{S:4}

The manifold $N$ constructed in Theorem \ref{T:main} is the total
space of a vector bundle, and is hence not compact.  We thus have
the following question:
\begin{Question}
Let $M^n$ be a smooth $n$-dimensional compact Riemannian manifold
and $(f_i)_{i=1}^k$ be a moving Parseval frame for $TM$ for some
$k\geq n$.  Does there exists a smooth $k$-dimensional compact
Riemannian manifold $N^k$ which has a moving orthonormal basis
$(e_i)_{i=1}^k$ such that $N^k$ contains $M^n$ as a submanifold and
$P_{T_x M}e_i(x)=f_i(x)$ for all $x\in M^n$ and $1\leq i\leq k$?
\end{Question}

If, $(f_i)_{i=1}^k$ is a Parseval frame for $\R^n$ and $k> m> n$,
then there exists a Parseval frame $(h_i)_{i=1}^k$ for
$\R^n\oplus\R^{m-n}$ such that $P_{\R^n}(h_i)=f_i$ for all $1\leq
i\leq m$.  Thus instead of dilating all the way to an orthonormal
basis for a $k$-dimensional Hilbert space, it is possible to dilate
to a Parseval frame for a $m$-dimensional Hilbert space.  This
motivates the following question in the vector bundle setting.

\begin{Question}
Let $k>m>n$ be integers, and let $\pi_1:E_1\rightarrow M$ be a rank
$n$ vector bundle over a paracompact manifold $M$ with a moving
Parseval frame $(f_i)_{i=1}^k$.  Does there exists a rank $m-n$
vector bundle $\pi_2:E_2\rightarrow M$ with a moving Parseval frame
$(g_i)_{i=1}^k$ so that $(f_i\oplus g_i)_{i=1}^k$ is a moving
Parseval frame for the vector bundle $\pi_1\oplus\pi_2:E_1\oplus
E_2\rightarrow M$?
\end{Question}

In \cite{BF}, it is proven that for every natural  numbers $k\geq
n$, there exists a tight frame $(f_i)_{i=1}^k$ of $\R^n$ such that
$\|f_i\|=1$ for all $1\leq i\leq k$, which they call a finite unit
tight frame or FUNTF. An explicit construction for FUNTFs is given
in \cite{DFKLOW}, and they are further studied in
\cite{BC},\cite{CFM}. FUNTFs are also of interest for applications
in signal processing, as they minimize mean squared error under an
additive noise model for quantization \cite{GK}. For a vector bundle
to have a moving FUNTF, it is necessary that it have a nowhere zero
section. Thus we have the following question concerning moving
FUNTFs.

\begin{Question}
Let $\pi:E\rightarrow N$ be a rank $n$ vector bundle over a
paracompact manifold $N$ such that $\pi$ has a nowhere zero section.
For what $k\geq n$ does $\pi$ have a moving tight frame
$(f_i)_{i=1}^k$ such that $\|f_i(x)\| = 1$ for all $x\in N$ and
$1\leq i \leq k$?
\end{Question}

These questions are general and potentially difficult, and so
solutions for certain cases would still be valuable.  For instance,
for general values of $k$ and $n$, it is unknown if the collection
of FUNTFs are connected \cite{DySt}.  A moving FUNTF of $k$ sections
for a rank $n$ vector bundle over the circle can be thought of as a
path in the collection of FUNTFs of $k$ vectors for $\R^n$.  Thus,
knowing whether or not a rank $n$ vector bundle over the circle has
a FUNTF of $k$ vectors will give insight into the problem of
determining the connected components of the collection of FUNTFs of
$k$ vectors for $\R^n$.


\begin{thebibliography}{99}

\bibitem[BF]{BF}J. J. Benedetto and M. Fickus, {\em Finite normalized tight frames}, Adv. Comput. Math.
{\bf 18} (2003), 357--385.


\bibitem[BC]{BC} B. G. Bodmann and P. G. Casazza,{\em The road to equal-norm Parseval
frames}, J. Fun. Anal., {\bf 258} (2010),  397--420.


\bibitem[BCPS]{BCPS} B. Bodmann, P. G. Casazza, V. Paulsen and D. Speegle,  {\em Spanning and independence
properties of frame partitions}, Proc. of the AMS, to appear.

\bibitem[CCLV]{CCLV} P. G. Casazza, O. Christensen, A. Lindner and R.
Vershynin, {\em Frames and the Feichtinger conjecture}, Proc. Amer.
Math. Soc. {\bf 133} (2005), 1025–1033.

\bibitem[CFM]{CFM} P. G. Casazza, M. Fickus and D. G. Mixon, {\em Auto-tuning unit norm
frames}, Appl. Comput. Harmon. Anal., to appear.



\bibitem[C]{C} O. Christensen, {\em An Introduction to Frames and Riesz
Bases}, Birkh$\ddot{a}$user, (2003).



\bibitem[DFKLOW]{DFKLOW} K. Dykema, D. Freeman, K. Kornelson, D. Larson, M. Ordower, and E.
Weber,{\em Ellipsoidal tight frames and projection decompositions of
operators}, Illinois J. Math. {\bf 48} (2004), no. 2, 477--489.

\bibitem[DGM]{DGM}I. Daubechies, A. Grossmann, and Y. Meyer, {\em Painless nonorthogonal
expansions}, J. Math. Phys. {\bf 27} (1986), 1271--1283.

\bibitem[DS]{DS}R. J. Dufflin, and A. C. Schaeffer,{\em A class of nonharmonic Fourier
series}, Trans. Amer. Math. Soc. {\bf 72} (1952), 341--366.

\bibitem[DySt]{DySt} K. Dykema and N. Strawn,{\em Manifold structure of spaces of spherical tight
frames},   Int. J. Pure Appl. Math. 28 (2006), 217--256.

\bibitem[HKLW]{HKLW}
D. Han, K. Kornelson, D. Larson, and E. Weber, \textit{Frames for
Undergraduates}, Amer. Math. Soc. {\bf 145}~(2007)

\bibitem[HL]{HL} D. Han and D. R. Larson, {\em Frames, bases and group representations}, Mem. Amer. Math.
Soc. 147 (2000), no. 697.

\bibitem[GK]{GK} V. K. Goyal, J. Kova$\tilde{c}$evi$\acute{c}$ and J. A. Kelner {\em Quantized frame
expansions with erasures}, ACHA {\bf 10} no. 3 (2001), 203--233.

\bibitem[K]{K}P. Kuchment, {\em Tight frames of exponentially decaying Wannier functions}, J.
Phys. A: Math. Theor. {\bf 42}, no. 2, (2009)

\bibitem[L]{L} J. M. Lee, {\em Introduction to Smooth Manifolds}, Springer, (2002).

\bibitem[W]{W} N. Weaver, {\em The Kadison-Singer Problem in discrepancy theory}, Discrete Math. {\bf
278} (2004), 227–-239.

 \end{thebibliography}
\end{document}